\documentclass{amsart}
\usepackage[numbers,square]{natbib}
\usepackage{graphicx}
\usepackage{amsmath}
\usepackage{amsthm}
\usepackage{amssymb}%

\newcommand{\E}{\mathbb E}
\newcommand{\R}{\mathbb{R}}
\newcommand{\N}{\mathbb{N}}

\newcommand{\Sph}{\mathbb{S}}
\renewcommand{\P}{\mathbb{P}}

\newcommand{\Var}{\mathop{\mathrm{Var}}\nolimits}
\newcommand{\Cov}{\mathop{\mathrm{Cov}}\nolimits}

\newcommand{\UU}{\mathcal{U}}

\newcommand{\Aa}{\mathcal{A}}
\newcommand{\BB}{\mathcal{B}}

\newcommand{\eqdistr}{\stackrel{\mathcal{D}}{=}}
\newcommand{\todistr}{\stackrel{\mathcal{D}}{\to}}

\theoremstyle{plain}
\newtheorem{theorem}{Theorem}[section]
\newtheorem{lemma}{Lemma}[section]

\newtheorem{proposition}{Proposition}[section]

\theoremstyle{definition}
\newtheorem{definition}{Definition}[section]
\newtheorem{problem}{Problem}[]
\newtheorem{remark}{Remark}[section]
\newtheorem{example}{Example}[section]

\theoremstyle{remark}

\newenvironment{step}[1][\textsc{Step}]
{\begin{trivlist}\item[\hskip \labelsep {\textsc{#1}}]}{\end{trivlist}}

\begin{document}

\title{Extremes of Independent Gaussian Processes}
\author{Zakhar Kabluchko}
\address{Institut f\"ur Mathematische Stochastik, Georg-August Universit\"at G\"ottingen,
Goldschmidtstra\ss e 7,
D--37077 G\"ottingen,
Germany
}
\email{kabluch@math.uni-goettingen.de}
\begin{abstract}
For every $n\in\N$, let $X_{1n},\ldots, X_{nn}$ be independent copies of a zero-mean Gaussian process $X_n=\{X_n(t), t\in T\}$. We describe all processes which can be obtained as limits, as $n\to\infty$, of the process $a_n(M_n-b_n)$, where $M_n(t)=\max_{i=1,\ldots,n} X_{in}(t)$ and $a_n, b_n$ are normalizing constants.
We also provide an analogous characterization for the limits of the process $a_nL_n$, where $L_n(t)=\min_{i=1,\ldots,n} |X_{in}(t)|$.
\end{abstract}


\maketitle

\section{Introduction}\label{sec:intro}
Suppose that we are given a large number $n$ of independent copies of some Gaussian process defined on an arbitrary set $T$. Let $M_n$ be the maximum of these processes, taken pointwise. The aim of this paper is to describe the class of processes which can be obtained as limits, as $n\to\infty$ and after suitable normalization, of the process $M_n$.
To state our problem more precise,
let $X_{1n},\ldots,X_{nn}$ be $n$ independent copies of a zero-mean Gaussian process $\{X_n(t), t\in T\}$
with
\begin{equation}
r_n(t_1,t_2):=\E[X_n(t_1)X_n(t_2)], \;\;\; \sigma^2_n(t):=r_n(t,t)>0.
\end{equation}
We define a process $\{M_n(t), t\in T\}$ by
\begin{equation}\label{eq:intro_Mnt}
M_n(t)=\max_{i=1,\ldots,n} X_{in}(t).
\end{equation}

\begin{problem}\label{prob:main}
Describe all sequences $X_n$ and all sequences of constants $a_n$, $b_n$ for which the process $a_n(M_n-b_n)$ converges as $n\to\infty$ to some nontrivial limit.
\end{problem}
Throughout, convergence of stochastic processes is understood as the weak convergence of their finite-dimensional distributions.

\citet{huesler_reiss89} proved that $M_n$ converges as $n\to\infty$ to a nontrivial limiting  process if the following two conditions are satisfied:
\begin{enumerate}
\item $\sigma_n(t)=1$ for all $t\in T$.
\item $\lim_{n\to\infty}\log n\cdot (1-r_n(t_1,t_2))$ exists in $(0,\infty)$ for all $t_1,t_2\in T$.
\end{enumerate}
We will complement the work of \citet{huesler_reiss89} by considering processes with \textit{non-unit variance}, providing  \textit{necessary and sufficient} conditions for convergence, stressing the role of \textit{negative-definite functions}, and  giving a different representation for the limiting processes.  This representation is related to a \textit{symmetry property} of certain exponential intensity Poisson point processes with Gaussian increments which is of independent interest.

Extremes of independent \textit{stationary} Gaussian processes on $\R^d$ have been studied by~\cite{kabluchko_schlather_dehaan07}. 
We will extend some of the  results of~\cite{kabluchko_schlather_dehaan07} to the non-stationary case.


We will also consider a similar problem for \textit{minima} (always understood in the sense of absolute value) of independent Gaussian processes. With the above assumptions, we define a process $\{L_n(t), t\in T\}$ by
\begin{equation}
L_n(t)=\min_{i=1,\ldots,n} |X_{in}(t)|.
\end{equation}
\begin{problem}\label{prob:min}
Describe all sequences $X_n$ and all sequences of constants $a_n$, $b_n$ for which the process $a_n(L_n-b_n)$ converges as $n\to\infty$ to some nontrivial limit.
\end{problem}

The paper is organized as follows. In Section~\ref{sec:neg_def} we state Theorem~\ref{theo:def_proc} which is needed to define the class of limiting processes for both maxima and minima of independent Gaussian processes. We solve Problem~\ref{prob:main} and Problem~\ref{prob:min} in Section~\ref{sec:max} and Section~\ref{sec:min}.  In Section~\ref{sec:examples} we consider maxima and minima of independent fractional Brownian motions, and construct a new class of $\alpha$-stable processes. The proof of Theorem~\ref{theo:def_proc} is given in Section~\ref{sec:proof_constr}.

\section{A family of Gauss--Poisson point processes}\label{sec:neg_def}
\subsection{Negative-definite kernels}

We start by recalling some definitions related to negative-definite kernels, see~\cite[Chapter~3]{berg_etal_book}.
A kernel on a set $T$ is a mapping from $T\times T$ to $\R$.
A kernel $\Gamma:T\times T\to [0,\infty)$ is called \textit{negative definite} if for every $n\in\N$, $t_1,\ldots,t_n\in T$ and $a_1,\ldots,a_n\in \R$ satisfying $\sum_{i=1}^n a_i=0$, we have
\begin{equation}\label{eq:neg_def}
\sum_{i=1}^n \sum_{j=1}^n a_ia_j \Gamma(t_i,t_j)\leq 0.
\end{equation}
It is well known that $d(t_1,t_2):=\sqrt{\Gamma(t_1,t_2)}$ defines a pseudo-metric on $T$ (which means that $d(t_1,t_1)=0$, $d(t_1,t_2)=d(t_2,t_1)$, and $d(t_1,t_3)\leq d(t_1,t_2)+d(t_2,t_3)$ for every $t_1,t_2,t_3\in T$, but note that, in general, $d(t_1,t_2)=0$ does not imply that $t_1=t_2$).  Further, a pseudo-metric $d$ on $T$ arises in this way for some negative-definite kernel $\Gamma$ if and only if the pseudo-metric space $(T, d)$ is isometrically embeddable into a Hilbert space.
\begin{example}
If $\{W(t), t\in T\}$ is a zero-mean Gaussian process, then its \textit{incremental variance}
$$
\Gamma(t_1,t_2):=\Var(W(t_1)-W(t_2))
$$
is a negative-definite kernel on $T$.
\end{example}

It will be convenient to use the following definition.
\begin{definition}\label{def:neg_def_ext}
A kernel $\Gamma:T\times T\to [0,\infty]$ is called \textit{negative definite in the extended sense} if there is a disjoint decomposition $T=\cup_{\alpha\in \Aa} T_{\alpha}$ such that the following two conditions hold:
\begin{enumerate}
\item $\Gamma(t_1, t_2)\neq +\infty$ if and only if there is $\alpha\in \Aa$ such that $t_1\in T_{\alpha}$ and $t_2\in T_{\alpha}$;
\item for every $\alpha\in \Aa$, the restriction of $\Gamma$ to $T_{\alpha}\times T_{\alpha}$ is negative definite in the usual sense.
\end{enumerate}
\end{definition}
Note that the above conditions show that $\Gamma$ determines the decomposition $T=\cup_{\alpha\in \Aa}T_{\alpha}$ uniquely. The next simple proposition establishes a link between positive-definite and negative-definite kernels.
\begin{proposition}\label{prop:neg_def}
For each $n\in\N$, let $\{X_n(t), t\in T\}$ be a zero-mean, unit-variance Gaussian process with covariance function $r_n(t_1,t_2):=\E[X_n(t_1)X_n(t_2)]$. Suppose that there exists a  sequence $z_n$ with $\lim_{n\to\infty} z_n=+\infty$ such that for all $t_1,t_2\in T$,
\begin{equation}\label{eq:cond_hr_gen}
\Gamma(t_1,t_2):=\lim_{n\to\infty} z_n(1-r_n(t_1,t_2))\in [0,\infty].
\end{equation}
Then the kernel $\Gamma$ is negative definite in the extended sense.
\end{proposition}
\begin{proof}
First we show that if  for some $t_1,t_2,t_3\in T$, both $\Gamma(t_1,t_2)$ and $\Gamma(t_2,t_3)$ are finite, then $\Gamma(t_1,t_3)$ is finite too. If the process $\{X_n(t), t\in T\}$ is defined on some probability space $(\Omega,\P)$, then we can define an embedding of $T$ into $L^2(\Omega)$ by $t\mapsto X_n(t)$. This embedding induces a pseudometric $d_n$ on $T$ defined by
$$
d_n^2(t_1,t_2)=2(1-r_n(t_1,t_2)).
$$
Now suppose that $\Gamma(t_1,t_2)$ and $\Gamma(t_2,t_3)$ are finite, which means that
$$
\lim_{n\to\infty} z_n^{1/2}  d_n(t_1,t_2)<\infty,\;\;\;
\lim_{n\to\infty} z_n^{1/2}  d_n(t_2,t_3)<\infty.
$$
By the triangle inequality, $d_n(t_1,t_3)\leq d_n(t_1,t_2)+d_n(t_2,t_3)$. It follows that
$$
\limsup_{n\to\infty} z_n^{1/2}  d_n(t_1,t_3)<\infty.
$$
This  implies that $\Gamma(t_1,t_3)$ is finite too.

Now we define an equivalence relation ``$\sim$'' on $T$ by declaring $t_1\sim t_2$ if and only if $\Gamma(t_1,t_2)\neq +\infty$.
Let $T=\cup_{\alpha\in \Aa}T_{\alpha}$ be the decomposition of $T$ into equivalence classes induced by the relation ``$\sim$''.
To complete the proof, we need to show that for every $\alpha\in\Aa$, the restriction of $\Gamma$ to $T_{\alpha}\times T_{\alpha}$ is negative definite in the usual sense. Take some $k\in\N$, $t_1,\ldots,t_k\in T_{\alpha}$ and $a_1,\ldots,a_k\in\R$ satisfying $\sum_{i=1}^ka_i=0$. Then, by~\eqref{eq:cond_hr_gen},
$$
\sum_{i=1}^k\sum_{j=1}^k a_ia_j \Gamma(t_i,t_j)=
\lim_{n\to\infty} z_n \sum_{i=1}^k\sum_{j=1}^k a_ia_j (1-r_n(t_i,t_j)).
$$
Since $\sum_{i=1}^ka_i=0$, we have
$$
\sum_{i=1}^k\sum_{j=1}^k a_ia_j \Gamma(t_i,t_j)=
-\lim_{n\to\infty} z_n  \sum_{i=1}^k\sum_{j=1}^k a_ia_j r_n(t_i,t_j).
$$
Now, the right-hand side is $\leq 0$ since the kernel $(t',t'')\mapsto r_n(t',t'')$, being the covariance function of the process $X_n$, is positive definite. This verifies~\eqref{eq:neg_def} and completes the proof.
\end{proof}

\subsection{A family of Gauss--Poisson point processes}
In the next theorem we construct a family of exponential intensity Poisson processes related to each other by Gaussian increments. This construction will be used to define limiting processes for both maxima and minima of independent Gaussian processes.
\begin{theorem}\label{theo:def_proc}
Fix $\lambda\in\R$ and a negative-definite kernel $\Gamma$ on a set $T$.  Let $\{U_i, i\in\N\}$ be a Poisson point process on $\R$ with intensity $e^{-\lambda u}du$, and let $W_i$, $i\in\N$, be independent copies of any zero-mean Gaussian process  $\{W(t), t\in T\}$ with incremental variance $\Gamma$.
Denote by $\sigma^2(t)$ the variance of $W(t)$. Define a function $\mathcal U_i:T\to\R$ by
\begin{equation}\label{eq:def_X_gamma}
\UU_i(t)=U_i+W_i(t)-\lambda \sigma^2(t)/2.
\end{equation}
Then the law of the random family of functions $\{\UU_i, i\in \N\}$
depends only on $\Gamma$ and $\lambda$(and does not depend on $\sigma^2$).
\end{theorem}
\begin{remark}\label{rem:def_Ws}
For a given negative-definite kernel $\Gamma$, there are many processes $W$ with incremental variance $\Gamma$. For example, for every $s\in T$ there is a unique in law Gaussian process $W^{(s)}$ with incremental variance $\Gamma$ such that additionally,  $W^{(s)}(s)=0$. The covariance function of this process is given by
\begin{equation}\label{eq:def_Ws}
\E[W^{(s)}(t_1)W^{(s)}(t_2)]=\frac 12 \left(\Gamma(t_1,s)+\Gamma(t_2,s)-\Gamma(t_1,t_2)\right).
\end{equation}
\end{remark}
\begin{remark}
We consider $\{\UU_i, i\in \N\}$ as a Poisson point process on the space $\R^T$ endowed with the product $\sigma$-algebra.
\end{remark}
\begin{remark}
For every fixed $t\in T$, $\{\mathcal U_i(t), i\in\N\}$ is a Poisson point process on $\R$ with intensity $e^{-\lambda u}du$. We may view the map $t\mapsto \{\mathcal U_i(t), i\in\N\}$ as a family of dependent Poisson point processes on $\R$ related to each other via Gaussian jumps.
\end{remark}
\begin{remark}
If $T=\R^d$ and $W$ is a process with \textit{stationary increments}, the statement of Theorem~\ref{theo:def_proc} can be deduced  from~\cite{kabluchko_schlather_dehaan07}. In this case, the law of the family of functions $\{\UU_i, i\in\N\}$ is translation invariant. The proof in the general case will be given in Section~\ref{sec:proof_constr}.
\end{remark}


\section{Maxima of independent Gaussian processes}\label{sec:max}
The main result of this section is Theorem~\ref{theo:main} which solves Problem~\ref{prob:main}. Before we can state and prove it, we need to define the class of limiting processes in Section~\ref{sec:constr_main} and to recall a result of \citet{huesler_reiss89} in a slightly generalized form and with a different representation of the limiting process in Section~\ref{sec:huesler_main}.
\subsection{Construction of limiting processes}\label{sec:constr_main}
Let $\Gamma$ be a negative-definite kernel on a set $T$.  Let $\{U_i, i\in\N\}$ be a Poisson point process on $\R$ with intensity $e^{-u}du$, and let $W_i$, $i\in\N$, be independent copies of a zero-mean Gaussian process $\{W(t), t\in T\}$ with incremental variance $\Gamma$ and variance $\sigma^2$. Then we set
\begin{equation}\label{eq:def_M_Gamma}
M_{\Gamma}(t)=\max_{i\in\N} (U_i+W_i(t)-\sigma^2(t)/2).
\end{equation}
\begin{remark}\label{rem:sym_main}
By Theorem~\ref{theo:def_proc} with $\lambda=1$, the law of $M_{\Gamma}$ depends only on $\Gamma$.
\end{remark}
\begin{remark}
A standard calculation with Poisson point processes, see~\cite{resnick_book}, shows that the finite-dimensional distributions of $M_{\Gamma}$ are as follows: for every $t_1,\ldots,t_k\in T$ and $y_1,\ldots,y_k\in\R$, we have
\begin{eqnarray}
\lefteqn{\P[M_{\Gamma}(t_1)\leq y_1,\ldots, M_{\Gamma}(t_k)\leq y_k]}\label{eq:fin_dim}\\
&=&
\exp\left(- \int_{\R}\P[\exists j: u+W(t_j)-\sigma^2(t_j)/2> y_j]e^{-u}du\right)\nonumber.
\end{eqnarray}
Note that for every fixed $t\in T$ , $M_{\Gamma}(t)$ is distributed according to the unit Gumbel distribution function $e^{-e^{-y}}$.
\end{remark}
\begin{remark}\label{rem:max_stab}
By construction, the process $M_{\Gamma}$ is max-stable, see~\cite{haan84}. This means that  for every $k\in \N$,  the process $\{\max_{i=1,\ldots,k}M_{\Gamma}^{(i)}(t), t\in T\}$, where $M_{\Gamma}^{(1)},\ldots, M_{\Gamma}^{(k)}$ are independent copies of $M_{\Gamma}$, has the same law as $\{M_{\Gamma}(t)+\log k, t\in T\}$.
\end{remark}

We need to slightly extend the definition of $M_{\Gamma}$ given above. Let $\Gamma$ be a negative-definite kernel in the \textit{extended sense} on a set $T$ with the corresponding decomposition $T=\cup_{\alpha\in \Aa}T_{\alpha}$. Then we denote by $\{M_{\Gamma}(t), t\in T\}$ the stochastic process with the following properties:
\begin{enumerate}
\item for every $\alpha\in\Aa$, the process $\{M_{\Gamma}(t), t\in T_{\alpha}\}$ is defined as in~\eqref{eq:def_M_Gamma}.
\item for $\alpha,\beta\in\Aa$ with $\alpha\neq \beta$, the processes $\{M_{\Gamma}(t), t\in T_{\alpha}\}$ and $\{M_{\Gamma}(t), t\in T_{\beta}\}$ are independent.
\end{enumerate}

\subsection{Sufficient conditions for convergence}\label{sec:huesler_main}
The next theorem is a slight modification of a result of \citet{huesler_reiss89}. Closely related results were obtained by~\cite{brown_resnick77}, \cite{kabluchko_schlather_dehaan07}. Let $u_n$ be a sequence such that
\begin{equation}\label{eq:asympt_un}
\sqrt{2\pi}u_ne^{u_n^2/2}\sim n,\;\;\; n\to\infty.
\end{equation}
Equivalently, one can take
\begin{equation}\label{eq:def_unvn}
u_n=\sqrt{2\log n}-\frac{(1/2) \log\log n +\log (2\sqrt{\pi})+o(1)}{\sqrt{2\log n}},\;\;\; n\to\infty.
\end{equation}
\begin{theorem}\label{theo:hr}
For every $n\in\N$, let $X_{1n},\ldots, X_{nn}$ be independent copies of a zero-mean, unit-variance Gaussian process $\{X_n(t), t\in T\}$ with covariance function $r_n(t_1,t_2):=\E[X_n(t_1)X_n(t_2)]$ such that for every $t_1,t_2\in T$,
\begin{equation}\label{eq:cond_hr}
\Gamma(t_1,t_2):=\lim_{n\to\infty} 4\log n\cdot (1-r_n(t_1,t_2))\in [0,\infty].
\end{equation}
Let $M_n(t)=\max_{i=1,\ldots,n}X_{in}(t)$. Then the process $u_n(M_n-u_n)$ converges as $n\to\infty$ to the process $M_{\Gamma}$.
\end{theorem}
\begin{remark}
By Proposition~\ref{prop:neg_def}, the kernel $\Gamma$ is negative definite in the extended sense. Hence, the process $M_{\Gamma}$ is well defined.
\end{remark}
\begin{remark}
\citet{huesler_reiss89} proved Theorem~\ref{theo:hr} assuming that $T$ is finite and that $\Gamma(t_1,t_2)<\infty$ for all $t_1,t_2\in T$, but they gave a different expression for the limiting process.
\end{remark}
\begin{proof}[Proof of Theorem~\ref{theo:hr}.]
Suppose first that $\Gamma(t_1,t_2)<\infty$ for all $t_1,t_2\in T$.
We will use a method similar to that of~\cite{huesler_reiss89} (see also~\cite{pickands69a}). Let $s\in T$ be arbitrary. It will be convenient to set $u_n(y)=u_n+u_n^{-1}y$ for $y\in\R$.
Conditioned on the event $A_n(y):=\{X_n(s)=u_n(y)\}$, we have
\begin{equation}
\E[u_n(X_n(t)-u_n)|A_n(y)]
=u_n(u_n(y)r_n(t,s)-u_n)
=yr_n(t,s)+u_n^2(r_n(t,s)-1).
\end{equation}
Taking into account~\eqref{eq:def_unvn} and~\eqref{eq:cond_hr}, we obtain
\begin{equation}\label{eq:lim_exp}
\lim_{n\to\infty} \E[u_n(X_n(t)-u_n)|A_n(y)]=y-\frac 12\Gamma(t,s).
\end{equation}
Further, by the well-known formula for the conditional covariance,
\begin{equation}
\Cov [u_n(X_n(t_1)-u_n),u_n(X_n(t_2)-u_n)| A_n(y)]=u_n^2(r_n(t_1,t_2)- r_n(t_1,s)r_n(t_2,s)).
\end{equation}
Together with~\eqref{eq:def_unvn} and~\eqref{eq:cond_hr} this implies that
\begin{equation}\label{eq:lim_cov}
\lim_{n\to\infty} \Cov [u_n(X_n(t_1)-u_n),u_n(X_n(t_2)-u_n)| A_n(y)]=\frac 12 (\Gamma(t_1,s)+\Gamma(t_2,s)-\Gamma(t_1,t_2)).
\end{equation}
It follows from~\eqref{eq:lim_exp} and~\eqref{eq:lim_cov} that for every fixed $y\in\R$, we have the following convergence of conditional processes
\begin{equation}\label{eq:conv_cond}
\{u_n(X_n(t)-u_n) | A_n(y) , t\in T\}  \todistr  \{y+W^{(s)}(t)-\frac 12\Gamma(t,s), t\in T \},\;\;\; n\to\infty.
\end{equation}

We are going to compute the finite-dimensional distributions of the process $u_n(M_n-u_n)$ in the limit $n\to\infty$.
Fix some $k\in\N$, $t_1,\ldots,t_k\in T$, and $y_1,\ldots,y_k\in\R$.
The density $f_n$ of the random variable $u_n(X_n(s)-u_n)$ is given by
$$
f_n(x)
=\frac{1}{\sqrt{2\pi}u_ne^{u_n^2/2}} e^{-x} e^{-\frac{x^2}{2u_n^2 }}.
$$
Conditioning on the event $A_n(y)$, we obtain
\begin{eqnarray*}
\lefteqn{\P[\exists j: X_n(t_j)>u_n(y_j)]}\\
&=&
\frac{1}{\sqrt{2\pi}u_ne^{u_n^2/2}}
\int_{\R} \P[\exists j: u_n(X_n(t_j)-u_n)> y_j|A_n(y)]e^{-y}e^{-\frac{y^2}{2u_n^2}}dy.
\end{eqnarray*}
It follows from~\eqref{eq:asympt_un}, \eqref{eq:conv_cond}, and a standard argument justifying the use of the dominated convergence theorem that
\begin{equation}\label{eq:wspom1}
\lim_{n\to\infty}n\P[\exists j: X_n(t_j)>u_n(y_j)]
=\int_{\R} \P\left[\exists j: y+W^{(s)}(t_j)-\frac 12 \Gamma(t_j,s)>y_j\right]e^{-y}dy.
\end{equation}
By the definition of $M_n$, we have
\begin{equation}\label{eq:wspom4}
\P[\forall j: u_n(M_n(t_j)-u_n)\leq y_j]
=(1-\P[\exists j: X_n(t_j)> u_n(y_j)])^n.
\end{equation}
It follows from~\eqref{eq:wspom1} and~\eqref{eq:wspom4} that
\begin{eqnarray}
\lefteqn{\lim_{n\to\infty}\P[\forall j: u_n(M_n(t_j)-u_n)\leq y_j]} \label{eq:probab_Mn}\\
&=&
\exp\left(-\int_{\R} \P\left[\exists j: y+W^{(s)}(t_j)-\frac 12\Gamma(t_j,s)>y_j\right]e^{-y}dy\right). \nonumber
\end{eqnarray}
Comparing this with~\eqref{eq:fin_dim}, we obtain the statement of the theorem under the restriction that $\Gamma(t_1,t_2)<\infty$ for all $t_1,t_2\in T$. (Note that by Remark~\ref{rem:sym_main}, we can take $W=W^{(s)}$ on the right-hand side of~\eqref{eq:fin_dim}.)

To prove the general case, let $\Gamma$ be a negative-definite kernel in the extended sense, with the corresponding decomposition $T=\cup_{\alpha \in A} T_{\alpha}$. Take $t_1,\ldots,t_k\in T$ and $y_1,\ldots,y_k\in \R$.
Let $I_{\alpha}$ be the set of all $i=1,\ldots,k$ with the property that $t_i\in T_{\alpha}$. By the Bonferroni inequality,
\begin{equation}\label{eq:bonfer}
S_n^{(1)}-S_n^{(2)}\leq \P[\exists j: X_n(t_j)> u_n(y_j)] \leq S_n^{(1)},
\end{equation}
where
\begin{eqnarray}
S_n^{(1)}&=&\sum_{\alpha \in A} \P[\exists j\in I_{\alpha}: X_n(t_j)> u_n(y_j)],\\
S_n^{(2)}&=&\sum_{\substack{i',i''=1,\ldots,k\\i'\nsim i''}} \P[X_n(t_{i'})> u_n(y_{i'}),X_n(t_{i''})>u_n(y_{i''})].
\end{eqnarray}
(We write $i'\sim i''$ if $i'$ and $i''$ are contained in the same $I_{\alpha}$). For every $\alpha\in A$, choose $s_{\alpha}\in T_{\alpha}$ arbitrarily. On the one hand, the above proof, see~\eqref{eq:wspom1}, shows that for every $\alpha\in A$,
\begin{eqnarray*}
\lefteqn{\lim_{n\to\infty} n\P[\exists j\in I_{\alpha}: X_n(t_j)> u_n(y_j)]}\\
&=&
\int_{\R} \P\left[\exists j\in I_{\alpha}: y+W^{(s_{\alpha})}(t_j)-\frac 12\Gamma(t_j,s_{\alpha})>y_j\right]e^{-y}dy.
\end{eqnarray*}
On the other hand, it follows from Lemma~\ref{lem:sibuya} below that we have $\lim_{n\to\infty}nS_n^{(2)}=0$.
It follows from~\eqref{eq:bonfer} that
\begin{eqnarray*}
\lefteqn{\lim_{n\to\infty}\P[\forall j: M_n(t_j)\leq u_n(y_j)]}\\
&=&
\exp\left(-\sum_{\alpha\in A} \int_{\R} \P\left[\exists j\in I_{\alpha}: y+W^{(s_{\alpha})}(t_j)-\frac 12\Gamma(t_j,s_{\alpha})>y_j\right]e^{-y}dy\right).
\end{eqnarray*}
This completes the proof since it follows from~\eqref{eq:fin_dim} and Section~\ref{sec:constr_main} that the right-hand side is equal to $\P[\forall j: M_{\Gamma}(t_j)\leq y_j]$.
\end{proof}
\begin{lemma}\label{lem:sibuya}
Let $X_n$, $n\in\N$,  be a sequence of Gaussian processes as in Theorem~\ref{theo:hr}. Suppose that for some $t_1,t_2\in T$, $\Gamma(t_1,t_2)=\infty$.
Then for every $y_1,y_2\in\R$,
$$
p_n:=n\P[X_n(t_i)>u_n+u_n^{-1}y_i, i=1,2]\to 0, \;\;\; n\to\infty.
$$
\end{lemma}
\begin{proof}
Fix $\lambda>0$. Let $\{X_n^{(\lambda)}(t), t\in\{t_1,t_2\}\}$ be a zero-mean, unit-variance Gaussian vector with covariance $\rho_n:=1-\lambda/\log n$. Recalling~\eqref{eq:cond_hr} and taking into account the fact that $\Gamma(t_1,t_2)=\infty$, we obtain that $r_n(t_1,t_2)<\rho_n$ for sufficiently large $n$. By Slepian's comparison lemma, see~\cite[Corollary~4.2.3]{leadbetter_book},
\begin{equation}\label{eq:slepian}
p_n \leq n \P[X_n^{(\lambda)}(t_1)>u_n+u_n^{-1}y_1, X_n^{(\lambda)}(t_2)>u_n+u_n^{-1}y_2].
\end{equation}
The limit of the right-hand side of~\eqref{eq:slepian} as $n\to\infty$ was computed by~\cite{huesler_reiss89}. It follows that
$$
\limsup_{n\to\infty} p_n\leq  e^{-y_1}+e^{-y_2}-\Phi\left(\lambda+\frac{y_2-y_1}{2\lambda}\right)e^{-y_1}-\Phi\left(\lambda+\frac{y_1-y_2}{2\lambda}\right)e^{-y_2},
$$
where $\Phi$ is the standard normal distribution function.
Since the above is true for every $\lambda>0$, we obtain the statement of the lemma by letting $\lambda\to\infty$.
\end{proof}

\subsection{Necessary and sufficient conditions for convergence}\label{sec:solution_main}
The next theorem solves Problem~\ref{prob:main}.
To exclude trivial cases, we adopt the following definition: A process $\{M(t), t\in T\}$ is called \textit{nondegenerate} if there is $t\in T$ such that $M(t)$ is not a.s.\ constant.
\begin{theorem}\label{theo:main}
For every $n\in\N$, let $X_{1n},\ldots, X_{nn}$ be independent copies of a zero-mean Gaussian process $\{X_n(t), t\in T\}$ with
\begin{equation}\label{eq:def_rn_sigma_n}
r_n(t_1,t_2):=\E[X_n(t_1)X_n(t_2)], \;\;\; \sigma^2_n(t):=r_n(t,t)>0.
\end{equation}
Define a process $\{M_n(t), t\in T\}$ by $M_n(t)=\max_{i=1,\ldots,n} X_{in}(t)$.
Then:
\begin{enumerate}
\item \label{p:main_1} there exist sequences $a_n>0$ and $b_n$ such that the process $a_n(M_n-b_n)$ converges as $n\to\infty$ to a nondegenerate limit $M$, if and only if, the following two conditions hold:
\begin{enumerate}
\item \label{p:main_1a} the following limit exists in $[0,\infty]$ for every $t_1,t_2\in T$:
\begin{equation}\label{eq:main_gamma}
\Gamma(t_1,t_2):=4\lim_{n\to\infty}\log n \cdot \left(1-\frac{r_n(t_1,t_2)}{\sigma_n(t_1)\sigma_n(t_2)}\right);
\end{equation}
\item \label{p:main_1b} the following limit exists in $\R$ for every $t_1,t_2\in T$:
\begin{equation}\label{eq:main_kappa}
\kappa(t_1,t_2):=2\lim_{n\to\infty} \log n \left(1-\frac{\sigma_n(t_1)}{\sigma_n(t_2)}\right);
\end{equation}
\end{enumerate}
\item
\label{p:main_2} if Condition~\ref{p:main_1a} holds, then the kernel $\Gamma$ is negative definite in the extended sense;
\item
\label{p:main_3}  if Conditions~\ref{p:main_1a} and~\ref{p:main_1b} are satisfied, then we can choose
\begin{equation}\label{eq:def_anbn}
a_n=u_n/\sigma_n(t_0),\;\;\;
b_n=\sigma_n(t_0)u_n,
\end{equation}
where $t_0\in T$ is fixed and $u_n$ is as in~\eqref{eq:def_unvn}. In this case, the limiting process $\{M(t),t\in T\}$ has the same law as $\{M_{\Gamma}(t)+\kappa(t_0,t), t\in T\}$, where $M_{\Gamma}$ is as in Section~\ref{sec:constr_main}.
\end{enumerate}
\end{theorem}
\begin{remark}
If we additionally assume that $\sigma_n\equiv 1$, then Condition~\ref{p:min_1b} is trivially fulfilled with $\kappa(t_1,t_2)=0$.
\end{remark}
We will need the following lemma from~\cite{kabluchko_schlather_dehaan07}, see Lemma~21 there.
\begin{lemma}\label{lem:conv_hr}
Let the assumptions of Theorem~\ref{theo:main} be satisfied and, additionally, $\sigma_n\equiv 1$. Take some $t_1,t_2\in T$. Then the law of the bivariate random vector $(u_n(M_n(t_1)-u_n), u_n(M_n(t_2)-u_n))$
converges weakly as $n\to\infty$  if and only if the following limit exists in $[0,\infty]$:
\begin{eqnarray}\label{eq:huesler_cond}
c:=\lim_{n\to\infty}\log n\cdot (1 - r_n(t_1,t_2)).
\end{eqnarray}
\end{lemma}

\begin{proof}[Proof of Theorem~\ref{theo:main}.]
Suppose that the process $a_n(M_n-b_n)$ converges as $n\to\infty$ to some nondegenerate limit $M$. First we prove the existence of the limit in~\eqref{eq:main_kappa}. The random variable $\sigma_n^{-1}(t)M_n(t)$ has the distribution of the maximum of $n$ independent standard Gaussian random variables. Thus, the random variable $u_n\sigma_n^{-1}(t)(M_n(t)-\sigma_n(t)u_n)$ converges as $n\to\infty$ weakly to the Gumbel distribution, see~\cite[Theorem~1.5.3]{leadbetter_book}. On the other hand, since we assume the process $M$ to be nondegenerate, for some $t\in T$, the random variable $a_n(M_n(t)-b_n)$ has a nondegenerate limiting distribution as $n\to\infty$. By the convergence to types lemma, see~\cite[Theorem~1.2.3]{leadbetter_book}, there exist constants $a(t)>0$ and $b(t)\in\R$ such that
\begin{equation}\label{eq:conv_to_types}
\lim_{n\to\infty} \frac{u_n\sigma_n^{-1}(t)} {a_n}=a(t) ,\;\;\; 
\lim_{n\to\infty} a_n(\sigma_n(t)u_n-b_n)=b(t).     
\end{equation}
Inserting $t=t_1$ and $t=t_2$ into the second equation in~\eqref{eq:conv_to_types} and taking the difference, we obtain
$$
\lim_{n\to\infty} a_n u_n(\sigma_n(t_1)-\sigma_n(t_2))=b(t_1)-b(t_2).
$$
Thus,
\begin{equation}\label{eq:wspom2}
\lim_{n\to\infty} u_na_n\sigma_n(t_1)\left(1-\frac{\sigma_n(t_2)}{\sigma_n(t_1)}\right)=b(t_1)-b(t_2).
\end{equation}
Recall that by~\eqref{eq:conv_to_types}, $a_n\sigma_n(t_1)\sim u_n/a(t)$,  and by~\eqref{eq:def_unvn}, $u_n^2\sim 2\log n$ as $n\to\infty$. Applying this to~\eqref{eq:wspom2}, yields Condition~\ref{p:min_1b}.

We prove that Condition~\ref{p:min_1a} is satisfied.  Take some $t_1,t_2\in T$ and assume that $t_1\neq t_2$, since otherwise, the limit in~\eqref{eq:main_gamma} is evidently $0$.
Define a unit-variance Gaussian process $\{X_n'(t), t\in T\}$ by $X_n'(t)=X_n(t)/\sigma_n(t)$.
Let $M_n'(t)=M_n(t)/\sigma_n(t)$. We claim that the process $u_n(M_n'-u_n)$ converges as $n\to\infty$ to a nondegenerate limit. Recall that we assume that the process $a_n(M_n-b_n)$ converges to $M$. We may write
\begin{equation}\label{eq:wspom3}
\{u_n(M_n'(t)-u_n), t\in T\}\eqdistr \{u_n\sigma_n^{-1}(t)(M_n(t)-\sigma_n(t)u_n), t\in T\}.
\end{equation}
It follows from~\eqref{eq:conv_to_types} that the process on the right-hand side of~\eqref{eq:wspom3} converges to $\{a(t)(M(t)-b(t)), t\in T\}$. The covariance function of $X_n'$ is given by
\begin{equation}\label{eq:rn_prime}
r_n'(t_1,t_2)=r_n(t_1,t_2)/\sigma_n(t_1)\sigma_n(t_2).
\end{equation}
Condition~\ref{p:min_1a} follows then from Lemma~\ref{lem:conv_hr}.

The above proves the ``only if'' statement of part~\ref{p:main_1}.  Since the ``if'' statement of part~\ref{p:main_1} will follow from part~\ref{p:main_3}, we proceed to the proof of part~\ref{p:main_2} of the theorem. Recall that $X_n'(t)=X_n(t)/\sigma_n(t)$. In view of~\eqref{eq:rn_prime}, we may rewrite~\eqref{eq:main_gamma} as
$$
\Gamma(t_1,t_2)=\lim_{n\to\infty} 4\log n\cdot (1-r_n'(t_1,t_2)).
$$
Applying Proposition~\ref{prop:neg_def} to the process $X_n'$, we obtain the statement of part~\ref{p:main_2}.

We prove part~\ref{p:main_3} of the theorem. Suppose that Conditions~\ref{p:main_1a} and~\ref{p:main_1b} hold. By Theorem~\ref{theo:hr}, the process $u_n(M_n'-u_n)$ converges to $M_{\Gamma}$ as $n\to\infty$. With $a_n$ and $b_n$ as in~\eqref{eq:def_anbn} we have
$$
\left\{a_n(M_n(t)-b_n), t\in T\right\}\eqdistr
\left\{\frac{\sigma_n(t)}{\sigma_n(t_0)}u_n(M_n'(t)-u_n)-u_n^2\left(1-\frac{\sigma_n(t)}{\sigma_n(t_0)}\right), t\in T\right\}.
$$
Note that by~\eqref{eq:main_kappa},
$$
\lim_{n\to\infty}\frac{\sigma_n(t)}{\sigma_n(t_0)}=1,\;\;\; \lim_{n\to\infty} u_n^2\left(1-\frac{\sigma_n(t)}{\sigma_n(t_0)}\right)=\kappa(t_0,t).
$$
This completes the proof of part~\ref{p:main_3} of the theorem.
\end{proof}

\section{Minima of independent Gaussian processes}\label{sec:min}
In this section we solve Problem~\ref{prob:min} using a method similar to that used in Section~\ref{sec:max}.
\subsection{Construction of limiting processes}\label{sec:constr_min}
We start by defining the class of limiting processes. Let $\Gamma$ be a negative-definite kernel on a set $T$.  Let $\{U_i, i\in\N\}$ be a Poisson point process on $\R$ with Lebesgue measure as intensity, and let $W_i$, $i\in\N$, be independent copies of a zero-mean Gaussian process $\{W(t), t\in T\}$ with incremental variance $\Gamma$ and variance $\sigma^2$. Then we define a process $\{L_{\Gamma}(t), t\in T\}$ by
\begin{equation}\label{eq:def_L_Gamma}
L_{\Gamma}(t)=\min_{i\in\N} |U_i+W_i(t)|.
\end{equation}
\begin{remark}
By Theorem~\ref{theo:def_proc} with $\lambda=0$, the law of $L_{\Gamma}$ depends only on $\Gamma$.
\end{remark}
\begin{remark}
It follows from the basic properties of Poisson point processes that the finite-dimensional distributions of $L_{\Gamma}$ are as follows: for every $t_1,\ldots,t_k\in T$ and $y_1,\ldots,y_k>0$, we have
\begin{equation}\label{eq:fin_dim_min}
\P[L_{\Gamma}(t_1)>y_1,\ldots, L_{\Gamma}(t_k)>y_k]=\exp\left(-\int_{\R}\P[\exists j: |u+W(t_j)|\leq  y_j]du \right).
\end{equation}
For every fixed $t\in T$, $L_{\Gamma}(t)$ has an exponential distribution with mean $1/2$.
\end{remark}
\begin{remark}
The process $L_{\Gamma}$ is max-infinitely divisible, which means (see~\cite{resnick_book} for a definition) that for every $k\in\N$, we can represent $L_{\Gamma}$ as a pointwise minimum of $k$ independent identically distributed processes. This follows from the fact that the Poisson point process on $\R$ with unit intensity can be represented as a union of $k$ independent Poisson point processes with constant intensity $1/k$. Note also that the property of max-infinite divisibility is weaker than that of max-stability, cf.\ Remark~\ref{rem:max_stab}.
\end{remark}

Exactly as in Section~\ref{sec:constr_main}, the above definition of the process $L_{\Gamma}$  can be extended to the case when $\Gamma$ is a negative-definite kernel in the \textit{extended sense}.

\subsection{Sufficient condition for convergence}
The next theorem is an analogue of Theorem~\ref{theo:hr} for minima (in the sense of absolute value) of independent  Gaussian processes. Let
\begin{equation}\label{eq:def_wn}
w_n=(2\pi)^{-1/2}n.
\end{equation}
\begin{theorem}\label{theo:min_hr}
For every $n\in\N$, let $X_{1n},\ldots, X_{nn}$ be independent copies of a zero-mean, unit-variance Gaussian process $\{X_n(t), t\in T\}$ with covariance function $r_n(t_1,t_2):=\E[X_n(t_1)X_n(t_2)]$ such that for every $t_1,t_2\in T$,
\begin{equation}\label{eq:cond_hr_min}
\Gamma(t_1,t_2):=\frac{1}{\pi}\lim_{n\to\infty} n^2(1-r_n(t_1,t_2))\in[0,\infty].
\end{equation}
Let $L_n(t)=\min_{i=1,\ldots,n}|X_{in}(t)|$. Then the process $w_nL_n$ converges as $n\to\infty$ to the process $L_{\Gamma}$.
\end{theorem}
\begin{remark}
By Proposition~\ref{prop:neg_def}, the kernel $\Gamma$ is negative-definite in the extended sense. Hence, the process $L_{\Gamma}$ is well-defined.
\end{remark}
\begin{proof}
The idea of the proof is analogous to that used in the proof of Theorem~\ref{theo:hr}. First suppose that $\Gamma(t_1,t_2)<\infty$ for every $t_1,t_2\in T$.
Fix some $s\in T$. Conditioned on the event $A_n(y):=\{X_n(s)=w_n^{-1}y\}$, we have
\begin{equation}\label{eq:lim_exp_min}
\lim_{n\to\infty} \E[w_n X_n(t)|A_n(y)]= \lim_{n\to\infty}y r_n(t,s)=y.
\end{equation}
Further, for every $t_1,t_2\in T$,
$$
\Cov [w_nX_n(t_1),w_n X_n(t_2)| A_n(y)]=w_n^2(r_n(t_1,t_2)- r_n(t_1,s)r_n(t_2,s)).
$$
It follows that
\begin{equation}\label{eq:lim_cov_min}
\lim_{n\to\infty} \Cov [w_n X_n(t_1),w_n X_n(t_2)| A_n(y)]=\frac 12(\Gamma(t_1,s)+\Gamma(t_2,s)-\Gamma(t_1,t_2)).
\end{equation}
It follows from~\eqref{eq:lim_exp_min} and~\eqref{eq:lim_cov_min} that for every $y\in \R$, we have the following convergence of conditioned stochastic processes:
\begin{equation}\label{eq:conv_cond_min}
\{w_nX_n(t) | A_n(y), t\in T\}\todistr\{y+W^{(s)}(t), t\in T\}, \;\;\; n\to\infty,
\end{equation}
where $W^{(s)}$ is as in Remark~\ref{rem:def_Ws}.

Take $k\in\N$, $t_1,\ldots,t_k\in T$, and $y_1,\ldots,y_k\geq 0$. The density $f_n$ of the random variable $w_nX_n(s)$ is given by
$$
f_{n}(y)=\frac{1}{\sqrt {2\pi}w_n}e^{-\frac{y^2}{2w_n^2}}.
$$
Conditioning on the event $A_n(y)$, we get
\begin{equation}\label{eq:wspom5}
\P[\exists j: w_n|X_n(t_j)|\leq y_j]=\frac{1}{\sqrt{2\pi}w_n}  \int_{\R} \P[\exists j: w_n|X_n(t_j)|\leq y_j|A_n(y)]e^{-\frac{y^2}{2w_n^2}}dy.
\end{equation}
It follows from~\eqref{eq:conv_cond_min} and~\eqref{eq:wspom5} that
$$
\lim_{n\to\infty}n\P[\exists j: w_n|X_n(t_j)|\leq  y_j]=\int_{\R} \P[\exists j: |y+W^{(s)}(t_j)|\leq y_j]dy.
$$
Again, we have omitted the standard justification of the use of the dominated convergence theorem. Recalling that $L_n$ is the minimum of $n$ independent copies of $|X_n|$, we obtain
$$
\lim_{n\to\infty}\P[\forall j: w_n L_n(t_j)> y_j]=\exp\left(-\int_{\R} \P[\exists j: |y+W^{(s)}(t_j)|\leq y_j]dy\right).
$$
Comparing this with~\eqref{eq:fin_dim_min} (we may take $W=W^{(s)}$ there) completes the proof of the theorem under the restriction that $\Gamma(t_1,t_2)$ is finite for all $t_1,t_2\in T$. The general case can be handled exactly as in the proof of Theorem~\ref{theo:hr} using Lemma~\ref{lem:sibuya_min} below.
\end{proof}
\begin{lemma}\label{lem:sibuya_min}
Let $X_n$, $n\in\N$,  be a sequence of Gaussian processes as in Theorem~\ref{theo:min_hr}. Suppose that for some $t_1,t_2\in T$, $\Gamma(t_1,t_2)=\infty$.
Then for every $y_1,y_2>0$,
$$
p_n:=n\P[w_n|X_n(t_i)|<y_i, i=1,2]\to 0, \;\;\; n\to\infty.
$$
\end{lemma}
\begin{proof}
The density $g_n$ of the bivariate vector $(w_nX_n(t_1),w_nX_n(t_2))$ is given by
$$
g_{n}(z_1,z_2)=\frac{1}{2\pi w_n^2 \sqrt{1-r_n(t_1,t_2)^2}}\exp\left(-\frac{z_1^2+z_2^2-2r_n(t_1,t_2)z_1z_2}{2w_n^2(1-r_n(t_1,t_2)^2)}\right).
$$
It follows from~\eqref{eq:huesler_cond_min} together with $\Gamma(t_1,t_2)=\infty$ that the term $\exp(\ldots)$ on the right-hand side converges to $1$ as $n\to\infty$.  Furthermore, $\lim_{n\to\infty}ng_{n}(z_1,z_2)=0$.
Thus,
$$
p_n=\int_{-y_1}^{y_1}\int_{-y_2}^{y_2}ng_n(z_1,z_2)dz_1dz_2\to 0,\;\;\; n\to\infty.
$$
This completes the proof of the lemma.
\end{proof}

\subsection{Necessary and sufficient conditions}
The next theorem solves Problem~\ref{prob:min}.
\begin{theorem}\label{theo:min}
For every $n\in\N$, let $X_{1n},\ldots, X_{nn}$ be independent copies of a zero-mean Gaussian process $\{X_n(t), t\in T\}$ with
\begin{equation}\label{eq:def_rn_sn_min}
r_n(t_1,t_2):=\E[X_n(t_1)X_n(t_2)], \;\;\; \sigma^2_n(t):=r_n(t,t)>0.
\end{equation}
Define a process $\{L_n(t), t\in T\}$ by $L_n(t)=\min_{i=1,\ldots,n} |X_{in}(t)|$. Then:
\begin{enumerate}
\item \label{p:min_1} there exists a sequence $a_n$ such that the process $a_nL_n$ converges as $n\to\infty$ to some nondegenerate limit $L$, if and only if, the following conditions hold:
\begin{enumerate}
\item \label{p:min_1a} the following limit exists in $[0,\infty]$ for every $t_1,t_2\in T$:
\begin{equation}\label{eq:min_gamma}
\Gamma(t_1,t_2):=\frac{1}{\pi}\lim_{n\to\infty} n^2 \left(1-\frac{r_n(t_1,t_2)}{\sigma_n(t_1)\sigma_n(t_2)}\right);
\end{equation}
\item \label{p:min_1b} the following limit exists in $(0,\infty)$ for every $t_1,t_2\in T$:
\begin{equation}\label{eq:min_kappa}
\kappa(t_1,t_2):=\lim_{n\to\infty}\frac{\sigma_n(t_1)}{\sigma_n(t_2)};
\end{equation}
\end{enumerate}
\item \label{p:min_2} if Condition~\ref{p:min_1a} holds, then the kernel $\Gamma$ is negative definite in the extended sense;
\item \label{p:min_3}  if Conditions~\ref{p:min_1a} and~\ref{p:min_1b} are satisfied, then we can choose
\begin{equation}\label{eq:def_an_min}
a_n=(2\pi)^{-1/2}\sigma_n^{-1}(t_0) n
\end{equation}
for some fixed $t_0$. In this case, the limiting process $\{L(t), t\in T\}$ has the same law as  $\{\kappa(t,t_0) L_{\Gamma}(t), t\in T\}$, where $L_{\Gamma}$ is as in Section~\ref{sec:constr_min}.
\end{enumerate}
\end{theorem}
\begin{remark}
Considering a more general normalization $a_n(L_n-b_n)$ will not lead to any essentially new processes by the convergence to types lemma.
\end{remark}
We will need the following two lemmas.
\begin{lemma}\label{lem:conv_hr_min}
Let the assumptions of Theorem~\ref{theo:min_hr} be satisfied and, additionally, $\sigma_n\equiv 1$. Take some $t_1,t_2\in T$. Then the law of the bivariate random vector $(w_nL_n(t_1), w_nL_n(t_2))$
converges weakly as $n\to\infty$ if and only if the following limit exists in $[0,\infty]$:
\begin{eqnarray}\label{eq:huesler_cond_min}
c:=\lim_{n\to\infty}n^2(1 - r_n(t_1,t_2)).
\end{eqnarray}
\end{lemma}
\begin{proof}
The ``if'' part of the lemma follows from Theorem~\ref{theo:min_hr}. To prove the ``only if'' part, suppose that the limit on the right-hand side of~\eqref{eq:huesler_cond_min} does not exist, which means that the sequence $n^2(1 - r_n(t_1,t_2))$ has at least two different accumulation points $c_1,c_2$ in $[0,\infty]$. It follows from Theorem~\ref{theo:min_hr} that the sequence of vectors $(w_nL_n(t_1), w_nL_n(t_2))$ has at least two weak accumulation points, and these points are different, as it can be seen from the explicit formula~\eqref{eq:fin_dim_min}. This contradiction completes the proof.
\end{proof}
\begin{lemma}\label{lem:exp_half}
If $Z_i$, $i\in\N$, are independent  standard normal variables, then the random variable $w_n\min_{i=1,\ldots,n}|Z_i|$, where $w_n$ is as in~\eqref{eq:def_wn}, converges as $n\to\infty$ to the exponential distribution with mean $1/2$.
\end{lemma}
\begin{proof}
Let $y>0$. Since the density of $Z_1$ is equal to $1/\sqrt{2\pi}$ at $0$, we have
$$
\lim_{n\to\infty}n\P[|Z_1|\leq w_n^{-1}y]=\lim_{n\to\infty} \frac{2w_n^{-1}ny}{\sqrt{2\pi}}= 2y.
$$
It follows that
$$
\lim_{n\to\infty}\P\left[w_n\min_{i=1,\ldots,n}|Z_i|>y\right]=\lim_{n\to\infty}(1-\P[|Z_1|\leq w_n^{-1}y])^n=e^{-2y}.
$$
This completes the proof.
\end{proof}
\begin{proof}[Proof of Theorem~\ref{theo:min}.]
We prove the ``only if'' statement of part~\ref{p:min_1} of the theorem. Suppose that the process $a_nL_n$ converges as $n\to\infty$ to some nondegenerate limit $L$. First we show that Condition~\ref{p:min_1b} holds. The random variable $\sigma_n^{-1}(t)L_n(t)$ has the distribution of the minimum (in the sense of absolute value) of $n$ independent standard Gaussian random variables. By Lemma~\ref{lem:exp_half}, the random variable $w_n\sigma_n^{-1}(t)L_n(t)$ converges as $n\to\infty$ weakly to the exponential distribution with mean $1/2$.
On the other hand, by the assumption, the random variable $a_nL_n(t)$ has a nondegenerate limiting distribution as $n\to\infty$. By the convergence to types lemma, see~\cite[Theorem~1.2.3]{leadbetter_book}, there exists a constant $a(t)>0$ such that
\begin{equation}\label{eq:conv_to_types_min}
\lim_{n\to\infty} \frac{w_n\sigma_n^{-1}(t)} {a_n}=a(t).
\end{equation}
Using this for $t=t_1$ and $t=t_2$ and taking the quotient, we obtain~\eqref{eq:min_kappa}.

We prove that Condition~\ref{p:min_1a} holds.  Take $t_1,t_2\in T$ with $t_1\neq t_2$ (otherwise, the limit in~\eqref{eq:min_gamma} is $0$).
Define a unit-variance Gaussian process $\{X_n'(t), t\in T\}$ by $X_n'(t)=X_n(t)/\sigma_n(t)$, and
let $L_n'(t)=L_n(t)/\sigma_n(t)$. We show that the process $w_nL_n'$ converges as $n\to\infty$. We have
$$
\{w_n L_n'(t), t\in T\}\eqdistr \left\{\frac{w_n \sigma_n^{-1}(t)}{a_n} a_n L_n(t), t\in T\right\}.
$$
It follows from~\eqref{eq:conv_to_types_min} that the process $w_nL_n'$ converges as $n\to\infty$ to $\{a(t)L(t), t\in T\}$. Condition~\ref{p:min_1a} follows then from Lemma~\ref{lem:conv_hr_min}.

To prove part~\ref{p:min_2} of the theorem, note that if Condition~\ref{p:min_1a} is satisfied, then $r_n'$, the covariance function of $X_n'$, satisfies
$$
\lim_{n\to\infty}  n^2\cdot (1-r_n'(t_1,t_2))=\pi\Gamma(t_1,t_2).
$$
Applying Proposition~\ref{prop:neg_def} to the process $X_n'$, we obtain the statement of part~\ref{p:min_2}.

We prove part~\ref{p:min_3}. If Condition~\ref{p:min_1a} is satisfied, then by Theorem~\ref{theo:min_hr}, the process $w_nL_n'$ converges to $L_{\Gamma}$. If $a_n$ is chosen as in~\eqref{eq:def_an_min}, then
$$
\{a_nL_n(t), t\in T\}\eqdistr \left\{\frac{\sigma_n(t)}{\sigma_n(t_0)}w_n L_n'(t), t\in T \right\}.
$$
It follows from Condition~\ref{p:min_1b} that the process on the right-hand side converges as $n\to\infty$ to the process $\{\kappa(t,t_0)L_{\Gamma}(t), t\in T\}$. This completes the proof.
\end{proof}

\section{Examples and applications}\label{sec:examples}
\subsection{Maxima of independent fractional Brownian motions}
Recall that a zero-mean Gaussian process $\{B(t), t\in\R\}$ is called fractional Brownian motion with index $\alpha$ if
\begin{equation}\label{eq:cov_fbm}
\Cov(B(t_1), B(t_2))=\frac{1}{2}\left(|t_1|^{\alpha}+|t_{2}|^{\alpha}- |t_1-t_2|^{\alpha}\right).
\end{equation}
The next theorem describes the limiting processes for maxima of independent fractional Brownian motions. For the usual Brownian motion (part~\ref{p:fr_2} of the theorem), it reduces to a result of \citet{brown_resnick77}(whose proof was based on the Markov property of the Brownian motion). On the other hand, maxima of independent \textit{stationary} Gaussian processes were studied by~\cite{kabluchko_schlather_dehaan07}.
We call a process $\{M(t), t\in T\}$ \textit{nontrivial} if it is nondegenerate (i.e., $M(t)$ is not a.s.\ constant for at least one $t\in T$), and, additionally, there are $t_1, t_2\in T$ such that $M(t_1)$ and $M(t_2)$ are not a.s.\ equal.
\begin{theorem}\label{theo:fr}
Let $B_i$, $i\in\N$, be independent copies of a fractional Brownian motion $\{B(t), t\geq 0\}$ with index $\alpha\in(0,2]$. Fix $t_0>0$ and, with $s_n>0$ to be specified later, define
$$
M_n(t)=\max_{i=1,\ldots, n} B_i(t_0+s_n t).
$$
\begin{enumerate}
\item \label{p:fr_1} Suppose that $\alpha\in(0,1)$. Define $u_n$ as in~\eqref{eq:def_unvn}, and let
$s_n=t_0/ (2\log n)^{1/\alpha}$.  Then  the process $u_n(M_n-u_n)$ converges as $n\to\infty$ to the process  $\{M_{\Gamma}(t),$ $t\in\R\}$, where $\Gamma(t_1,t_2)=|t_1-t_2|^{\alpha}$.
\item  \label{p:fr_2} Suppose that $\alpha=1$. Define $u_n$ as in~\eqref{eq:def_unvn}, and let $s_n=t_0/(2\log n)$. Then the process $u_n(M_n-u_n)$ converges as $n\to\infty$ to the process $\{M_{\Gamma}(t)+t/2, t\in\R\}$, where $\Gamma(t_1,t_2)=|t_1-t_2|$.
\item \label{p:fr_3} Suppose that $\alpha\in (1,2]$. Then there exist no sequences $a_n$, $b_n$, and $s_n>0$ for which the process $a_n(M_n-b_n)$ converges to a nontrivial limiting process.
\end{enumerate}
\end{theorem}
\begin{remark}
An additional rescaling represented by the sequence $s_n$ is needed to obtain a nontrivial limit, cf.~\cite{brown_resnick77, kabluchko_schlather_dehaan07}.
\end{remark}
\begin{remark}
Using methods of~\cite{kabluchko_schlather_dehaan07}, it can be shown that  the convergence in parts~\ref{p:fr_1} and~\ref{p:fr_2} is weak on $C([-A,A])$, the space of continuous functions on $[-A,A]$, for every $A>0$.
\end{remark}
\begin{proof}[Proof of Theorem~\ref{theo:fr}.]
We are going to apply Theorem~\ref{theo:main} with $X_n(t)=B(t_0+s_nt)$. First, let $s_n$ be an arbitrary sequence with $s_n\to 0$. The variance $\sigma_n^2(t):=\Var X_n(t)$ satisfies, by~\eqref{eq:cov_fbm},
\begin{equation}\label{eq:fbm_sigma_n}
\sigma_n(t)=(t_0+s_n t)^{\alpha/2}=t_0^{\alpha/2} \left(1+\frac{\alpha t}{2t_0}s_n+ o(s_n)\right), \;\;\; n\to\infty.
\end{equation}
Again using~\eqref{eq:cov_fbm}, we obtain for $r_n(t_1,t_2):=\Cov(X_n(t_1),X_n(t_2))$, the covariance function of $X_n$,
\begin{equation}\label{eq:fr_wspom1a}
r_n(t_1,t_2)=t_0^{\alpha}\left(1+\frac{\alpha (t_1+t_2)}{2t_0}s_n+o(s_n)\right)- \frac 12 s_n^{\alpha}|t_1-t_2|^{\alpha}, \;\;\; n\to\infty.
\end{equation}
It follows from~\eqref{eq:fbm_sigma_n} that for every $t_1,t_2\in\R$,
\begin{equation}\label{eq:fr_wspom1}
1-\frac{\sigma_n(t_1)}{\sigma_n(t_2)}=\frac{\alpha (t_2-t_1)}{2t_0} s_n+ o(s_n),\;\;\; n\to\infty.
\end{equation}

We prove part~\ref{p:fr_1} of the theorem. Suppose that $\alpha\in (0,1)$ and recall that $s_n=t_0/(2\log n)^{1/\alpha}$. Then $s_n=o(s_n^{\alpha})$ as $n\to\infty$, and an easy calculation based on~\eqref{eq:fbm_sigma_n} and~\eqref{eq:fr_wspom1a} yields
\begin{equation}\label{eq:fr_wspom3}
1-\frac{r_n(t_1,t_2)}{\sigma_n(t_1)\sigma_n(t_2)}=\frac {|t_1-t_2|^{\alpha}}{2t_0^{\alpha}} s_n^{\alpha}+o(s_n^{\alpha}),\;\;\; n\to\infty.
\end{equation}
Hence, Condition~\ref{p:main_1a} of Theorem~\ref{theo:main} is fulfilled with $\Gamma(t_1,t_2)=|t_1-t_2|^{\alpha}$. Further, it follows from~\eqref{eq:fr_wspom1} that Condition~\ref{p:main_1b} of Theorem~\ref{theo:main} is fulfilled with $\kappa(t_1,t_2)=0$. The statement of part~\ref{p:fr_1} follows.

We prove part~\ref{p:fr_2} of the theorem. Suppose that $\alpha=1$. Then $s_n^{\alpha}=s_n$ and thus, a calculation based on~\eqref{eq:fbm_sigma_n} and~\eqref{eq:fr_wspom1a} yields
\begin{equation}\label{eq:fr_wspom4}
1-\frac{r_n(t_1,t_2)}{\sigma_n(t_1)\sigma_n(t_2)}=
\frac {|t_1-t_2|}{2t_0} s_n+ o(s_n) ,\;\;\; n\to\infty.
\end{equation}
Hence, Condition~\ref{p:main_1a} of Theorem~\ref{theo:main} is fulfilled with $\Gamma(t_1,t_2)=|t_1-t_2|$. Further, it follows from~\eqref{eq:fr_wspom1} that Condition~\ref{p:main_1b} of Theorem~\ref{theo:main} is fulfilled with $\kappa(t_1,t_2)=(t_2-t_1)/2$. The statement of part~\ref{p:fr_2} follows.

Finally, we prove  part~\ref{p:fr_3} of the theorem. Assume that $\alpha\in (1,2]$ and that for some $a_n$, $b_n$, and $s_n>0$, the process $a_n(M_n-b_n)$ converges to some nontrivial limit $M$. First of all, note that $\lim_{n\to\infty}s_n=0$, since otherwise, Condition~\ref{p:main_1b} of Theorem~\ref{theo:main} is not fulfilled.
We have $s_n^{\alpha}=o(s_n)$ as $n\to\infty$, and an easy calculation based on~\eqref{eq:fbm_sigma_n} and~\eqref{eq:fr_wspom1a} shows that
\begin{equation}\label{eq:fr_wspom2}
1-\frac{r_n(t_1,t_2)}{\sigma_n(t_1)\sigma_n(t_2)}=o(s_n),\;\;\; n\to\infty.
\end{equation}
It follows from~\eqref{eq:fr_wspom1} that in order that Condition~\ref{p:main_1b} of Theorem~\ref{theo:main} be fulfilled, we need that $\lim_{n\to\infty}s_n\log n$ exists finitely. This implies that $s_n=O(1/\log n)$ as $n\to\infty$. It follows from~\eqref{eq:fr_wspom2} that in Condition~\ref{p:main_1a} of Theorem~\ref{theo:main}, we have $\gamma(t_1,t_2)=0$. Thus, $M(t_1)=M(t_2)$ for all $t_1,t_2\in T$, which means that the limiting process $M$ is trivial.
\end{proof}

\subsection{Minima of independent fractional Brownian motions}
The next theorem describes possible limits for minima of independent fractional Brownian motions. In contrast with Theorem~\ref{theo:fr},  there is no change in behavior at $\alpha=1$ here.
\begin{theorem}\label{theo:fbm_min}
Let $B_i$, $i\in\N$, be independent copies of a fractional Brownian motion $\{B(t), t\geq 0\}$ with index $\alpha\in (0,2]$. Fix $t_0>0$, and, with $s_n=t_0(2\pi/n)^{1/\alpha}$,  define
$$
L_n(t)=\min_{i=1,\ldots, n} |B_i(t_0+s_n t)|.
$$
Then  the process $(2\pi)^{-1/2} nL_n$ converges as $n\to\infty$ to the process  $\{L_{\Gamma}(t), t\in\R\}$, where $\Gamma(t_1,t_2)=|t_1-t_2|^{\alpha}$.
\end{theorem}
\begin{remark}
It can be shown that the convergence is weak on the space $C([-A,A])$ for every $A>0$.
\end{remark}
\begin{remark}
It is possible to obtain a generalization of Theorem~\ref{theo:min_hr}, Theorem~\ref{theo:min} and Theorem~\ref{theo:fbm_min} to $\chi^2$-processes. Recall that a $d$-dimensional $\chi^2$-process is defined as the Euclidian norm of an $\R^d$-valued process with i.i.d.\ Gaussian components. This, in particular, recovers the results of~\cite{penrose88,penrose91}, where minima of independent Bessel processes were studied. Note that the proofs of~\cite{penrose88,penrose91} are based on the Markov property of Bessel processes. Global minima of stationary $\chi^2$-processes were studied by~\cite{albin96}.
\end{remark}
\begin{proof}
We will apply Theorem~\ref{theo:min} to the process $X_n(t)=B(t_0+s_nt)$.  Since $\lim_{n\to\infty}s_n=0$, we have, for every $t\in\R$,
$$
\lim_{n\to\infty}\sigma_n(t)=|t_0|^{\alpha/2}.
$$
This implies that Condition~\ref{p:min_1b} of Theorem~\ref{theo:min} holds with $\kappa(t_1,t_2)=1$.  As in the proof of Theorem~\ref{theo:fr}, the variance and the covariance of $X_n$ satisfy~\eqref{eq:fbm_sigma_n} and~\eqref{eq:fr_wspom1a}. Hence,
$$
1-\frac{r_n(t_1,t_2)}{\sigma_n(t_1)\sigma_n(t_2)}=\frac {|t_1-t_2|^{\alpha}}{2t_0^{\alpha}} s_n^{\alpha}+o(s_n^{\alpha}),\;\;\; n\to\infty.
$$
It follows that Condition~\ref{p:min_1a} of Theorem~\ref{theo:min} holds with $\Gamma(t_1,t_2)=|t_1-t_2|^{\alpha}$.
\end{proof}

\subsection{$\alpha$-stable fields indexed by metric spaces}
Let $(T,d)$ be a pseudo-metric space admitting an isometric embedding into a Hilbert space. We are going to use Theorem~\ref{theo:def_proc} to construct a natural $\alpha$-stable process on $T$ whose dependence structure is determined by the metric $d$.
Recall that the kernel $\Gamma$ defined by $\Gamma(t_1,t_2):=d^2(t_1,t_2)$ is negative definite. Let $\{U_i, i\in\N\}$ be a Poisson point process on $\R$ with intensity $e^{-u}du$ and let $W_i$, $i\in\N$, be independent copies of any zero-mean Gaussian process  $\{W(t), t\in T\}$ with incremental variance $\Gamma$. Let also  $\sigma^2(t)=\Var W(t)$.
\begin{theorem}\label{theo:ex_alpha_stab}
For $\alpha\in(0,2)$, define a process $\{S_{\Gamma, \alpha}(t), t\in T\}$ by
\begin{equation}\label{eq:def_ex_stab}
S_{\Gamma, \alpha}(t)= \sum_{i\in\N}\left(e^{ \alpha (U_i+W_i(t)-\sigma^2(t)/2)}-b_i^{(\alpha)}\right),
\end{equation}
where
$$
b_i^{(\alpha)}=
\begin{cases}
0, & 0<\alpha<1,\\
\int_{1/i}^{1/(i-1)}x^{-2}\sin x dx,& \alpha=1,\\
\frac{\alpha}{\alpha-1}(i^{\frac{\alpha}{\alpha-1}}- (i-1)^{\frac{\alpha}{\alpha-1}}), & 1<\alpha<2.
\end{cases}
$$
Then the process $S_{\Gamma, \alpha}$ is an $\alpha$-stable process totally skewed to the right, and the law of $S_{\Gamma, \alpha}$ depends only on $\Gamma$  and $\alpha$ (and does not depend on $\sigma^2$).
\end{theorem}
\begin{proof}
The $\alpha$-stability of $S_{\Gamma,\alpha}$ follows from the series representation of multivariate $\alpha$-stable distributions, see~\cite[Theorem~3.10.1]{samorodnitsky_book}. The fact that $S_{\Gamma,\alpha}$ depends only on $\Gamma$ and $\alpha$  follows from Theorem~\ref{theo:def_proc}, since $S_{\Gamma,\alpha}$ can be viewed as a functional of the random family of functions $\{\mathcal U_i, i\in \N\}$.
\end{proof}
\begin{remark}
The above approach can be modified to construct \textit{symmetric} $\alpha$-stable processes.  To this end, supply the summands on the right-hand side of~\eqref{eq:def_ex_stab} with random signs.
\end{remark}
\begin{remark}
For two classes of $\alpha$-stable processes indexed by \textit{positive-definite} kernels, sub-Gaussian and harmonizable processes, see~\cite{samorodnitsky_book}. Examples of $\alpha$-stable processes indexed by certain metric spaces (which are different from ours) were given by~\cite{istas_06}.
In a particular case of stationary processes on $\R^d$, the possibility  to construct processes as in Theorem~\ref{theo:ex_alpha_stab} was mentioned in~\cite{kabluchko08b}.
\end{remark}
\begin{example}
Let $T=\Sph^d$ be the unit $d$-dimensional sphere and denote by $\rho$ the geodesic distance on $\Sph^d$. It is known, see e.g.~\cite{istas_05}, that the kernel $\Gamma(t_1,t_2):=\rho^{\beta}(t_1,t_2)$ is negative definite for $\beta\in(0,1)$. This leads to a family of max-stable (Section~\ref{sec:constr_main}) and a family of $\alpha$-stable (Theorem~\ref{theo:ex_alpha_stab}) processes indexed by $\alpha\in (0,2)$ and $\beta\in (0,1)$. Note that the laws of these processes are invariant with respect to the rigid rotations of the sphere. A similar construction is possible on a space of constant \textit{negative} curvature.
\end{example}


\section{Proof of Theorem~\ref{theo:def_proc}.}\label{sec:proof_constr}
\begin{step}[Step 1.]
Fix $s\in T$ and let $\{W^{(s)}(t), t\in T\}$  be a zero-mean Gaussian process defined as in Remark~\ref{rem:def_Ws}, i.e.
\begin{equation}\label{eq:def_Ws_1}
W^{(s)}(s)=0,\;\;\; \E[W^{(s)}(t_1)W^{(s)}(t_2)]=\frac 12 (\Gamma(t_1,s)+\Gamma(t_2,s)-\Gamma(t_1,t_2)).
\end{equation}
Let $\{\UU_i^{(s)}, i\in \N\}$ be the random collection of functions constructed as in Theorem~\ref{theo:def_proc} with $W=W^{(s)}$. More precisely, we define $\UU_i^{(s)}:T\to \R$ by
\begin{equation}\label{eq:def_u_i_s}
\UU_i^{(s)}(t)=U_i+W_i^{(s)}(t)-\lambda \Gamma(t,s)/2.
\end{equation}

In Step~1 of the proof we are going to show that the law of the random family of functions $\{\UU_i^{(s)},i\in\N\}$ does not depend on the choice of $s\in T$. To this end, we will modify an argument from~\cite{kabluchko_schlather_dehaan07}.

Take some $t_1,\ldots,t_k\in T$. For a set $B\subset \R^k$ and $x\in\R$, denote by $B+x$ the diagonally shifted set $B+(x,x,\ldots,x)$. Let $\P_{t_1,\ldots,t_k}^{(s)}$ be the law of the random vector
\begin{equation}\label{eq:vect}
(W^{(s)}(t_1)-\lambda\Gamma(t_1,s)/2,\ldots,W^{(s)}(t_k)-\lambda\Gamma(t_k,s)/2),
\end{equation}
considered as a probability measure on $\R^k$.  Then we can view the random collection of points
$$
\{(\UU_i^{(s)}(t_1),\ldots,\UU_i^{(s)}(t_k)), i\in \N\}
$$
as a Poisson point process on $\R^k$ whose intensity  $\Lambda_{t_1,\ldots,t_k}^{(s)}$ is given by
\begin{equation}\label{eq:Lambda_s}
\Lambda_{t_1,\ldots,t_k}^{(s)}(B)=\int_{\R} e^{\lambda x} \P_{t_1,\ldots,t_k}^{(s)}(B+x)dx, \;\;\; B\in \BB(\R^k).
\end{equation}
Here, $\BB(\R^k)$ is the $\sigma$-algebra of Borel subsets of $\R^k$.
Consider a measure $\mu_{t_1,\ldots,t_k}^{(s)}$ on the $(k-1)$-dimensional hyperplane $\{(x_i)_{i=1}^k\in\R^k: x_1=0\}$ in $\R^k$, defined for $A\in \BB(\R^k)$ by
\begin{equation}\label{eq:mu_s_A}
\mu_{t_1,\ldots,t_k}^{(s)}(A)=\int_{\R^{k}} e^{\lambda y_1} 1_{A}(0,y_2-y_1,\ldots,y_k-y_1)
d\P_{t_1,\ldots,t_k}^{(s)}(y_1,\ldots,y_k).
\end{equation}
Note that $\mu_{t_1,\ldots,t_k}^{(s)}(\R^k)= \E e^{\lambda W^{(s)}(t_1)-\lambda^2\Gamma(t_1,s)/2}=1$, and therefore,  $\mu_{t_1,\ldots,t_k}^{(s)}$ is a probability measure.
As in~\cite{kabluchko_schlather_dehaan07}, we can deduce from~\eqref{eq:Lambda_s} that
\begin{eqnarray}\label{eq:Lambda}
\Lambda_{t_1,\ldots,t_k}^{(s)}(B)=\int_{\R}e^{\lambda z}\mu_{t_1,\ldots,t_k}^{(s)}(B+z)dz,\;\;\; B\in\BB(\R^k).
\end{eqnarray}
Let $\psi_{t_1,\ldots,t_k}^{(s)}(u_1,\ldots,u_k)=\int_{\R^k} e^{\sum_{i=1}^k u_iy_i}d\mu_{t_1,\ldots,t_k}^{(s)}(y_1,\ldots,y_k)$ be the Laplace transform of the measure $\mu_{t_1,\ldots,t_k}^{(s)}$. It follows from~\eqref{eq:mu_s_A} that
\begin{eqnarray}
\lefteqn{\psi_{t_1,\ldots,t_k}^{(s)}(u_1,\ldots,u_k)}\label{eq:psi_varphi}\\
&=&\int_{\R^{k}} e^{\lambda y_1}
e^{u_2(y_2-y_1)+\ldots +u_k(y_k-y_1)} d\P_{t_1,\ldots,t_k}^{(s)}(y_1,\ldots,y_k)
\nonumber\\
&=&\varphi_{t_1,\ldots,t_k}^{(s)}\left(\lambda-\sum_{i=2}^k u_i, u_2,\ldots, u_k\right), \nonumber
\end{eqnarray}
where $\varphi_{t_1,\ldots,t_k}^{(s)}=\int_{\R^k} e^{\sum_{i=1}^k u_iy_i}d\P_{t_1,\ldots,t_k}^{(s)}(y_1,\ldots,y_k)$ is the Laplace transform of the measure
$\P_{t_1,\ldots,t_k}^{(s)}$.
Recall that the measure $\P_{t_1,\ldots,t_k}^{(s)}$ is Gaussian with known covariance and expectation, see~\eqref{eq:vect}. Hence,
\begin{eqnarray}
\lefteqn{\varphi_{t_1,\ldots,t_k}^{(s)}(u_1,\ldots, u_k)}\label{eq:laplace} \\
&=&
\exp\left(-\frac{\lambda}{2}\sum_{i=1}^k \Gamma(t_i,s) u_i+ \frac12 \sum_{i,j=1}^{k}(\Gamma(t_i,s)+\Gamma(t_j,s)-\Gamma(t_i,t_j))u_i u_j \right)\nonumber.
\end{eqnarray}
An elementary calculation based on~\eqref{eq:psi_varphi} and~\eqref{eq:laplace} shows that
\begin{eqnarray*}
\lefteqn{\psi_{t_1,\ldots,t_k}^{(s)}(u_1,\ldots,u_k)}\label{eq:psi} \\
&=&
\exp\left(-\frac{\lambda}2\sum_{i=2}^k \Gamma(t_i,t_1)u_i+ \frac12\sum_{i,j=2}^k(\Gamma(t_i,t_1)+\Gamma(t_j,t_1)-\Gamma(t_i,t_j))u_iu_j \right).\nonumber
\end{eqnarray*}
It follows that the measure $\mu_{t_1,\ldots,t_k}^{(s)}$ does not depend on $s\in T$. By~\eqref{eq:Lambda}, this implies that the measure $\Lambda_{t_1,\ldots,t_k}^{(s)}$ does not depend on $s$. Hence, the law of the random family of functions $\{\UU_i^{(s)}, i\in \N\}$ does not depend on the choice of $s\in T$.
\end{step}
\begin{remark}
An alternative way to prove the result of Step~1 is as follows. Let $\{X_n(t), t\in T\}$ be a zero-mean Gaussian process whose covariance function is given by
$$
r_n(t_1,t_2)=e^{-\Gamma(t_1,t_2)/(4\log n)},\;\;\; t_1,t_2\in T.
$$
Note that $r_n$ is indeed a valid covariance function by Schoenberg's theorem, see~\cite{berg_etal_book}. The sequence $X_n$ satisfies the assumptions of Theorem~\ref{theo:hr}. Taking some $s\in T$  and following the proof of Theorem~\ref{theo:hr}, we arrive at~\eqref{eq:probab_Mn}. Since the left-hand side of~\eqref{eq:probab_Mn} does not depend on $s$, it follows that the right-hand side of~\eqref{eq:probab_Mn} does not depend on $s$ either. This implies the statement of Step~1.
\end{remark}
\begin{step}[Step 2.]
Now let $\{W(t), t\in T\}$ be a general zero-mean Gaussian process with incremental variance $\Gamma$ and variance $\sigma^2$.
Our aim is to show that the law of the random family of functions $\{\UU_i, i\in \N\}$, where $\UU_i$ is defined as in Theorem~\ref{theo:def_proc}, i.e.
\begin{equation}\label{eq:def_u_i_3}
\UU_i(t)=U_i+W_i(t)-\lambda \sigma^2(t)/2,
\end{equation}
does not depend on $\sigma^2$.

Let $\bar T$ be an extension of the set $T$ by a point $\tau_0$, i.e.\ $\bar T=T\cup\{\tau_0\}$.
We extend the kernel $\Gamma$ to the set $\bar T$ as follows: let $\bar \Gamma(t_1,t_2)=\Gamma(t_1,t_2)$ if $t_1,t_2\in T$; $\bar \Gamma(\tau_0,t)=\bar \Gamma(t,\tau_0)=\sigma^2(t)$; and $\bar \Gamma(\tau_0,\tau_0)=0$.
Then $\bar \Gamma$ is a negative-definite kernel on $\bar T$. To see this, note that $\bar \Gamma$ is the incremental variance of the Gaussian process $\{\bar W^{(\tau_0)}(t), t\in \bar T\}$ defined by $\bar W^{(\tau_0)}(\tau_0)=0$, and $\bar W^{(\tau_0)}(t)=W(t)$ for $t\in T$.

Let $s\in \bar T$ be arbitrary. Then there is a unique zero-mean Gaussian process $\{\bar W^{(s)}(t), t\in \bar T\}$ with incremental variance $\bar \Gamma$ and $\bar W^{(s)}(s)=0$, see Remark~\ref{rem:def_Ws}.
Step~1 of the proof implies that the law of the random family of functions $\{\bar\UU_i^{(s)}, i\in\N\}$, where $\bar \UU_i^{(s)}:\bar T\to \R$ is defined by
\begin{equation}\label{eq:def_u_i_2}
\bar\UU_i^{(s)}(t)=U_i+\bar W_i^{(s)}(t)-\lambda \bar \Gamma(t,s)/2,
\end{equation}
does not depend on $s\in  T$ (here, $\bar W_i^{(s)}$, $i\in\N$, are independent copies of $\bar W^{(s)}$).
Note that the restriction of the process $\bar W^{(\tau_0)}$ to $T$ has the same law as $W$, and that for $s\in T$, the restriction of $\bar W^{(s)}$ to $T$ has the same law as $W^{(s)}$. Thus, for $s\in T$, the restriction of the family $\{\bar \UU_i^{(s)}, i\in\N\}$ to $T$ has the same law as the family $\{\UU_i^{(s)}, i\in\N\}$ defined in~\eqref{eq:def_u_i_s}, whereas the restriction of the family $\{\bar \UU_i^{(\tau_0)}, i\in\N\}$ to $T$ has the same law as the family $\{\UU_i, i\in \N\}$ defined in~\eqref{eq:def_u_i_3}.
It follows that the family $\{\UU_i, i\in \N\}$ has the same law as the family $\{\UU_i^{(s)}, i\in\N\}$, where $s\in T$ is arbitrary. Since the law of the second family does not depend on $\sigma^2$, we obtain the statement of the theorem.
\hfill$\Box$
\end{step}

\section*{Acknowledgements} The author is grateful to Martin Schlather for useful remarks.

\bibliographystyle{plainnat}
\bibliography{paper17bib}
\end{document}